\documentclass[a4paper,12pt]{article}
\usepackage{amssymb, amsmath, afterpage}
\usepackage{amsfonts}
\usepackage[dvips]{graphics}
\newcommand{\mybox}{\hfill $\square$}

\newcommand{\pli}[1]{\ensuremath{\overset{+}{\imath}}}

\newcommand{\sym}{\ensuremath{\mathrm{Sym}}}

\newcommand{\I}{\ensuremath{\mathcal{I}}}

\newcommand{\ip}[2]{\ensuremath{\langle{#1},{#2}\rangle}} %innerprod
\newcommand{\eps}{\varepsilon}

\newcommand{\qed}{\hfill $\square$ \\}

\newcommand{\zz}{\mathbb{Z}}

\newtheorem{thm}{Theorem}[section]
\newtheorem{lemma}[thm]{Lemma}         %and begin lemma etc for all this lot.
\newtheorem{prop}[thm]{Proposition}

\newenvironment{proof}{\paragraph{Proof}}{\mybox}
\begin{document}
    \title{\Large\textbf{Involution Products in Coxeter Groups}}
\date{}
      \author{S.B. Hart and P.J. Rowley\thanks{The authors wish to acknowledge partial support for this work from
      the Birkbeck Faculty of Social
      Science and the Manchester Institute for Mathematical Sciences (MIMS).
      }}
  \maketitle
\vspace*{-10mm}
\begin{abstract} \noindent For $W$ a Coxeter group, let $\mathcal{W} =
\{ w \in W \;| \; w = xy \; \mbox{where} \; x, y \in W \; \mbox{and} \; x^2 = 1 = y^2 \}$. If $W$ is finite,
then it is well known that $W = \mathcal{W}$. Suppose that $w \in \mathcal{W}$. Then
the minimum value of $\ell(x) + \ell(y) - \ell(w)$, where $x, y \in W$ with $w = xy$
and $x^2 = 1 = y^2$, is called the \textit{excess} of $w$ ($\ell$
is the length function of $W$). The main result established here is that $w$
is always $W$-conjugate to an element with excess equal to zero. (MSC2000: 20F55)

%It is well known that every element $w$ of a finite
%Coxeter group can be written as a product of at most two
%involutions. In general it is not possible to choose $x, y$ such
%that $w = xy$, $x^2 = y^2 = 1$, and $l(w) = l(x) + l(y)$. However
%one can choose $x, y$ such so that $L(w) = L(x) + L(y)$ (where
%$L(w)$ is the reflection length of $w$). Given all possible $x, y$
%with $x^2 = y^2 = 1$ and $xy = w$, the minimum value of $l(x) +
%l(y) - l(w)$ is called the excess of $w$; the reflection excess of
%$w$ is the minimum value over all such $x, y$ for which,
%additionally, $L(x) + L(y) = L(w)$. The purpose of this article is
%to investigate the properties of excess and reflection excess.

\end{abstract}

\section{Introduction}\label{intro}

The study of a Coxeter group $W$ frequently weaves together
features of its root system $\Phi$ and properties of its length
function $\ell$. The delicate interplay between $\ell(w_1w_2)$ and
$\ell(w_1) + \ell(w_2)$ for various $w_1, w_2 \in W$ is often to
be seen in investigations into the structure of $W$. Instances of
the additivity of the length function, that is $\ell(w_1w_2) =
\ell(w_1) + \ell(w_2)$, are of particular interest. For example,
if $W_J$ is
a standard parabolic subgroup of
$W$, then there is a set $X_J$ of so-called distinguished right coset
representatives for $W_J$ in $W$ with the property that $\ell(wx)
= \ell(w) + \ell(x)$ for all $w \in W_J$, $x \in X_J$
(\cite{humphreys}, Proposition 1.10). There is a parallel
statement to this for double cosets of two standard parabolic
subgroups of $W$ (\cite{gkpf}, Proposition 2.1.7). Also, when $W$ is finite it
possesses an element $w_0$, the longest element of $W$, for which
$\ell(w_0) = \ell(w) + \ell(ww_0)$ for all $w \in W$ (\cite{gkpf},
Lemma 1.5.3).\\

 Of the involutions (elements of order $2$) in $W$,
the reflections, and particularly the fundamental reflections,
more often than not play a major role in investigating $W$. This
is due to there being a correspondence between the reflections in
$W$ and the roots in $\Phi$. In general, involutions occupy a
special position in a group and it is sometimes possible to say
more about them than it is about other elements of the group. This
is true in the case of Coxeter groups. For example, Richardson
\cite{richardson} gives an effective algorithm for
parameterizing the involution conjugacy classes of a Coxeter
group. In something of the same vein we have the fact that an
involution can be expressed as a canonical product of reflections
(see Deodhar \cite{deodhar} and Springer \cite{springer}). \\

Suppose that $W$ is a Coxeter group (not necessarily finite or
even of finite rank), and put
$$ \mathcal{W} =
\{ w \in W \;| \; w = xy \; \mbox{where} \; x, y \in W \; \mbox{and} \; x^2 = 1 = y^2 \}.$$

That is, $\mathcal{W}$ is the set of strongly real elements of
$W$. For $w \in \mathcal{W}$ we define the \textit{excess} of $w$,
$e(w)$, by

$$e(w) = \min\{\ell(x) + \ell(y) - \ell(w) \; | \; w=xy, x^2 = y^2 =
1\}.$$

Thus $e(w) = 0$ is equivalent to there existing $x, y \in W$ with $x^2 = 1 = y^2$,
 $w = xy$ and $\ell(w) = \ell(x) + \ell(y)$. We shall call $(x, y)$, where $x, y$ are
involutions, a {\sl spartan pair} for $w$ if $w = xy$ with
$\ell(x) + \ell(y) - \ell(w) = e(w)$. As a small example take  $w
= (1234)$ in  $\sym(4) \cong W(A_3)$ (the Coxeter group of type
$A_3$). Then $\ell(w) = 3$ and $w$ can be written in four ways as
a product of involutions (see Table \ref{t1}).

 \begin{table}[h!bt]\begin{center}\begin{tabular}{ccr}
$x$&$y$&$\ell(x)+\ell(y)$\\
\hline $(13)$&$(14)(23)$&$3+6=9$\\
$(14)(23)$&$(24)$&$6+3=9$\\
$(24)$&$(12)(34)$&$3+2=5$\\
$(12)(34)$&$(13)$&$2+3=5$ \end{tabular}\label{t1}\end{center}\caption{$w = (1234) = xy$}\end{table} %

Thus $e(w) = 2$ with $((24), (12)(34))$ and $((12)(34), (13))$ being spartan pairs for $w$.
To give some idea of the
distribution of excesses we briefly mention two other examples. The
number of elements with excess 0, 2, 4, 6, 8 in $\sym(6) \cong
W(A_5)$ is, respectively, 489, 173, 46, 10, 2, the maximum excess
being 8. For $\sym(7) \cong W(A_6)$, the number of elements with
excess 0, 2, 4, 6, 8, 10, 12 is, respectively, 2659, 1519, 574,
228, 50, 8, 2 and here the maximum excess is 12.\\

In the case $W$ is finite we have $W = \mathcal{W}$, and so excess
is defined for every element of $W$. (Since $W$ is a direct
product of irreducible Coxeter groups, it suffices to check this
for $W$ an irreducible Coxeter group -- if $W$ is a Weyl group see
Carter \cite{carter}. The case when $W$ is a dihedral group is
straightforward to verify while types $H_3$ or $H_4$ may be
checked using \cite{magma}.) However if $W$ is infinite we can
have $W \ne \mathcal{W}$. This can be seen when $W$ is of type
$\widetilde{A_2}$. Then $W = HN$, the semidirect product of $N
=\{(\lambda_1,\lambda_2,\lambda_3)| \lambda_i \in \zz, \lambda_1 +
\lambda_2 + \lambda_3 = 0 \} \cong \zz \times \zz$ and $H \cong
W(A_2) \cong \sym(3)$ with $H$ acting on $N$ by permuting the
co-ordinates of $(\lambda_1,\lambda_2,\lambda_3)$. Let $g = (12)
\in H$ and $0 \ne \lambda \in \zz$, and set $w =
g(\lambda,\lambda,-2\lambda)$. Clearly $g$ and
$(\lambda,\lambda,-2\lambda)$ commute and
$(\lambda,\lambda,-2\lambda)$ has infinite order. Therefore $w$
has infinite order. If $w$ can be written as a product of two
involutions, then there exist $hm, kn \in W$ with $h, k \in H$,
$m, n \in N$ and $(hm)(kn) = w$. Therefore $h, k$ are self-inverse
elements of $H$ with $hk = (12)$. So one of $h$ or $k$ is $(12)$
and the other is the identity. But $N$ has no elements of order 2,
so either $hm$ or $kn$ is the identity, contradicting the fact
that $w$ is not an involution. So we conclude that $w$ cannot be
expressed as a
product of two involutions and hence $W \ne \mathcal{W}$.\\

%then there exists an involution $x$ in $W$ such that $w^x =
%w^{-1}$. Hence in $\overline{W} = W/N$,
%$\overline{w}^{\overline{x}} = \overline{w}^{-1} = \overline{w}$
%(as $\overline{w} = \overline{g}$). Thus $x \in  gN$ and so $x =
%gn$ for some $n \in N$. Therefore
%$$w^x = (g(\lambda,\lambda,-2\lambda))^{gn} =
%g^n(\lambda,\lambda,-2\lambda),$$ as $N$ is abelian. Now $x$ being
%an involution forces $n = (\mu,-\mu,0)$ for some $\mu \in \zz$.
%Consequently
%$$w^x = g^n(\lambda,\lambda,-2\lambda) = gg^{-1}g^n(\lambda,\lambda,-2\lambda) =
%g(2\mu,-2\mu,0)(\lambda,\lambda,-2\lambda)$$ $$= g(2\mu +
%\lambda,-2\mu + \lambda,-2\lambda).$$ Since $w^x = w^{-1}$, we
%deduce that $(2\mu + \lambda,-2\mu + \lambda,-2\lambda) =
%(-\lambda,-\lambda,2\lambda)$ which then yields $\lambda = 0$, a
%contradiction.

As we have observed a Coxeter group may have many elements with non-zero excess.
Nevertheless our main theorem shows the zero excess elements are ubiquitous
from a conjugacy class viewpoint.

\begin{thm}\label{exc=0}
Suppose $W$ is a Coxeter group, and let $w \in \mathcal{W}$.
Let $X$ denote the $W$-conjugacy class of $w$. Then there exists $w_{\ast} \in X$
such that $e(w_{\ast}) = 0$.
\end{thm}

Theorem \ref{exc=0} is proved in Section \ref{zero}, after gathering together a
number of preparatory results about Coxeter groups in Section \ref{back}. Also
some easy properties of excess are noted and, in Proposition \ref{12..n}, we
demonstrate that there are Coxeter groups in which elements can have arbitrarily
large excess.

\section{Background Results and Notation}\label{back}

Assume, for this section, that $W$ is a finite rank Coxeter group.
So, by its very definition, $W$ has a
presentation of the form

\[ W = \langle R \; | \; (rs)^{m_{rs}} = 1, r, s \in R\rangle \]
where $R$ is finite, $m_{rr} = 1$, $m_{rs} = m_{sr} \in
\mathbb{Z}^+\cup \{\infty\}$ and $m_{rs} \geq 2$ for $r, s \in R$
with $r\neq s$. The elements of $R$ are called the fundamental
reflections of $W$ and the rank of $W$ is the cardinality of $R$. The
length of an element $w$ of $W$, denoted by $\ell(w)$, is defined
to be
\[\ell(w) = \left\{ \begin{array}{l} \min\{ l \; | \; w=r_1r_2 \cdots r_l: r_i \in R\} {\mbox{ if $w \neq 1$}} \\
0 {\mbox{ if $w=1$}.} \end{array}\right.\]

Now let $V$ be a real vector space with basis $\Pi = \{ \alpha_r \;
| \; r \in R\}$. Define a symmetric bilinear form $\ip{ \; }{ \; }$
on $V$ by
\[\ip{\alpha_r}{\alpha_s} =  -\cos\left(\frac{\pi}{m_{rs}}\right)
.\]
\noindent where $r, s \in R$ and the $m_{rs}$ are as in the above
presentation of $W$. Letting $r, s \in R$ we define
\[ r\cdot\alpha_s = \alpha_s - 2\ip{\alpha_r}{\alpha_s}\alpha_r.\]
This then extends to an action of $W$ on $V$ which is both
faithful and respects the bilinear form $\ip{ \; }{ \; }$ (see 5.3
of \cite{humphreys}) and the elements of $R$ act as reflections
upon $V$. The module $V$ is called a reflection module for $W$ and
the following subset of $V$
\[\Phi = \{w\cdot\alpha_r \; | \; r \in R, w \in W\}\]
is the all important root system of $W$. Setting $\Phi^+ =
\{\sum_{r\in R}\lambda_r\alpha_r \in \Phi \; | \; \lambda_r \geq 0
{\mbox{ for all $r$}}\}$ and $\Phi^- = -\Phi^+$ we have the
fundamental fact that $\Phi$ is the disjoint union $\Phi^+
\dot\cup \Phi^-$ (see 5.4 -- 5.6 of \cite{humphreys}). The sets
$\Phi^+$ and $\Phi^-$ are referred to, respectively, as the
positive and negative roots of $\Phi$.

For $w \in W$ we define
\[N(w) = \{ \alpha \in \Phi^+ \; | \; w\cdot\alpha \in \Phi^- \}.\]
The connection between $\ell(w)$ and the root system of $W$ is contained
in our next lemma.
%mentioned above is that $l(w) = |N(w)|$ (see Section
%5.6 of \cite{humphreys}).

\begin{lemma}\label{millie}Let $w \in W$ and $r \in R$.
\begin{trivlist}
\item (i) If $\ell(wr) > \ell(w)$ then $w\cdot \alpha_r \in
\Phi^+$ and if $\ell(wr) < \ell(w)$ then $w\cdot \alpha_r \in
\Phi^-$. In particular, $\ell(wr) < \ell(w)$ if and only if
$\alpha_r \in N(w)$. \item (ii) $\ell(w) = |N(w)|$.
\end{trivlist}
\end{lemma}

\begin{proof}See Sections 5.4 and 5.6 of \cite{humphreys}.
\end{proof}

\begin{lemma} \label{lengthadd} Let $g, h \in W$. Then
$$N(gh) = N(h)\setminus [-h^{-1}N(g)] \cup h^{-1}[N(g)\setminus
N(h^{-1})].$$ %
Hence $\ell(gh) = \ell(g) + \ell(h) -2\mid N(g)\cap
N(h^{-1})\mid$.\end{lemma}

\begin{proof} Let $\alpha \in \Phi^+$. Suppose $\alpha \in N(h)$.
Then $gh\cdot \alpha = g\cdot (h\cdot \alpha)$ is negative if and
only if $-h\cdot\alpha \notin N(g)$. That is, $\alpha \notin
-h^{-1}N(g)$. Thus
$$N(gh) \cap N(h) = N(h)\setminus  -h^{-1}N(g).$$
If on the other hand $\alpha \notin N(h)$, then $gh \cdot \alpha
\in \Phi^{-}$ if and only if $h\cdot\alpha \in N(g)$. That is,
$\alpha \in \Phi^+ \cap h^{-1} N(g)$. Thus
$$N(gh) \setminus N(h) = h^{-1}[N(g)\setminus
N(h^{-1})].$$ Combining the two equations gives the expression for $N(gh)$ stated in
Lemma \ref{lengthadd}. Since $N(gh) \cap N(h)$ and $N(gh) \setminus N(h)$ are clearly
disjoint, using Lemma \ref{millie}(ii) we deduce that
$$\ell(gh) = \ell(g) + \ell(h) - 2\mid N(g)\cap N(h^{-1})\mid,$$
which completes the proof of the lemma.
\end{proof}\\

\begin{prop} \label{length} Let $w \in W$ and $r \in R$. If $\ell(rw)< \ell(w)$ and $\ell(wr)<\ell(w)$, then either
$rwr=w$ or $\ell(rwr) = \ell(w)-2$. \end{prop}

\begin{proof} See Exercise 3 in 5.8 of \cite{humphreys}.
\end{proof}\\

For $J \subseteq R$ define $W_J$ to be the subgroup of $W$
generated by $J$. Such a subgroup of $W$ is referred to as a
standard parabolic subgroup. Standard parabolic subgroups are
Coxeter groups in their own right with root system
\[\Phi_J = \{w\cdot\alpha_{r} \; | \; r \in J, w \in W_J\}\]
(see Section 5.5 of \cite{humphreys} for more on this). A
conjugate of a standard parabolic subgroup is called a parabolic
subgroup of $W$. Finally, a {\sl cuspidal} element of $W$ is an
element which is not contained in any proper parabolic subgroup of
$W$. Equivalently, an element is cuspidal if its $W$-conjugacy
class has empty intersection with all the proper standard
parabolic subgroups of $W$.

\begin{thm} \label{tits}  Let $0 \neq v \in V$.
Then the stabilizer of $v$ in $W$ is a parabolic subgroup of $W$.
Furthermore, if $v \in \Phi$, then the stabilizer of $v$ in $W$ is
a proper parabolic subgroup of $W$.
\end{thm}

\begin{proof}The fact that the stabilizer of $v$ is a parabolic subgroup is proved in Ch V \S 3.3 of
\cite{titsref}. If $v \in \Phi$, then $v = w\cdot \alpha_r$ for
some $r \in R, w \in W$ and hence $(wrw^{-1})\cdot v = -v$. Thus
the stabilizer of $v$ cannot be the whole of $W$, so is a proper
parabolic subgroup of $W$.
\end{proof}

\begin{thm}\label{involutions} Suppose that $w$ is an involution in $W$. Then there
exists $J \subseteq R$ such that $w$ is $W$-conjugate to $
w_J$, an element of $W_J$ which acts as $-1$ upon $\Phi_J$.
\end{thm}

\begin{proof} See Richardson \cite{richardson}.
\end{proof}\\

Next we give some easy properties of excess.

\begin{lemma}\label{basics} Let $ w \in \mathcal{W}$. Then the following
hold.
\begin{trivlist}
\item{(i)} If $w$ is an involution or the identity element, then $e(w) = 0$.
\item{(ii)} $e(w)$ is non-negative and even.%
 \item{(iii)} If $w = xy$ where $x$ and $y$ are involutions and $2|N(x) \cap N(y)| = e(w)$,
 then $(x,y)$ is a spartan pair for $w$.

\end{trivlist}\end{lemma}

\begin{proof} If $w$ is an involution or $w = 1$, then $w = w1$ with $w^2 = 1 = 1^2$,
whence $e(w) =0$. For (ii), suppose $x^2 = y^2 = 1$
and $xy = w$. Then, using Lemma \ref{lengthadd}, $\ell(w) = \ell(x)
+ \ell(y) - 2|N(x) \cap N(y)|$ and hence $\ell(x) + \ell(y) -
\ell(w)$ is even and (ii) follows. Part (iii) is also immediate from Lemma \ref{lengthadd}.
\end{proof}\\

We now have the tools needed for the proof of Theorem \ref{exc=0},
but before continuing with this we calculate the excess of the
element $(1 2 \cdots n)$ of $\sym(n)$. The aim of this is to show
that there are Coxeter groups in which elements may have
arbitrarily large excess. Before stating our next result we
require some notation. For $q$ a rational number $\lfloor
q\rfloor$ denotes the floor of $q$ (that is, the largest integer
less than or equal to $q$), and $\lceil q \rceil$ denotes the
ceiling  of $q$ (that is, the smallest integer
greater than or equal to $q$).\\

Let $n \geq 2$. Then $\sym(n)$ is isomorphic to the Coxeter group
$W(A_{n-1})$ of type $A$. If $W = W(A_{n-1}) \cong \sym(n)$, then
we set $R = \{(12), (23), \ldots, (n-1 \;\; n)\}$ and the set of
positive roots is $\Phi^+ = \{e_i - e_j \; | \; 1 \leq i < j \leq n\}$.
An alternative formulation of $\ell(w)$ for $w \in W$ in this case
is $\ell(w) = |\{(i,j) \; | \; 1 \leq i < j \leq n, w(i) > w(j)\}|$. For
$0 \leq k \leq n-1$, let $s_k$ be the longest element of
$\sym(\{1,2,\ldots, k\})$ and let $t_k$ be the longest element of
$\sym(\{k+1, k+2, \ldots, n\})$. A straightforward calculation shows that%
\begin{eqnarray*} s_k &=& (1\;\; k)(2\;\; k-1)\cdots (\lfloor \textstyle\frac{k}{2}\rfloor \;\; \lceil
\textstyle\frac{k}{2}\rceil + 1); {\mbox{ and}}\\
t_k &=& (k+1 \;\;\; n)(k+2 \;\;\; n-1)\cdots (\lfloor
\textstyle\frac{n-k}{2}\rfloor + k  \;\;\;\; \lceil
\textstyle\frac{n-k}{2}\rceil + k + 1).\end{eqnarray*} Note that
$s_0 = s_1 = t_{n-1} = 1$. Finally, for $0 \leq k \leq n-1$, set $y_k = s_kt_k$.\\

\begin{prop} \label{12..n} Let $w = (12\cdots n) \in W = W(A_{n - 1}) \cong \sym(n)$. Put
$\I_w = \{x \in W \; | \; x^2 = 1, w^{x} = w^{-1}\}.$ Then
\begin{trivlist}
\item{(i)} $\I_w = \{ y_k \; | \; 0 \leq k \leq n-1\}$.%
\item{(ii)} If $n$ is odd, then $(wy_{(n-1)/2}, y_{(n-1)/2})$ is
a spartan pair for $w$.%
\item{(iii)} If $n$ is even, then $(wy_{n/2}, y_{n/2})$ and
$(wy_{n/2-1}, y_{n/2-1})$ are both spartan pairs for $w$. %
\item{(iv)} $e(w) = \lfloor \frac{(n-2)^2}{2}\rfloor$.
\end{trivlist}
\end{prop}

\begin{proof} It is easy to check that $y_k \in
\I_w$ whenever $0 \leq k \leq n-1$. Since $|I_w| \leq |C_W(w)| =
n$, part (i) follows immediately.\\
Suppose $y \in \I_w$. Write $\eps_y(w) = \ell(wy) + \ell(y) -
\ell(w)$, so that $e(w) = \min\{\eps_y(w) \; | \; y \in \I_w\}$.
We have \begin{eqnarray*}%
\eps_y(w) &=& \ell(wy) + \ell(y) - \ell(w)\\
&=& \ell(w) + \ell(y) - 2|N(y)\cap N(w)| + \ell(y) - \ell(w)\\
&=& 2(\ell(y) - |N(y)\cap N(w)|)\end{eqnarray*}%
Now $N(w) = \{e_i-e_n \; | \; 1\leq i <n\}$. Therefore $$|N(y) \cap N(w)|
= |\{i \; | \; y(i)  > y(n) \}| = n-y(n).$$ %
 Let $y = y_k$ for some $0 \leq k \leq n-1$. We have $\ell(y_k) =
\ell(s_k) + \ell(t_k) = \frac{k(k-1)}{2} +
\frac{(n-k)(n-k-1)}{2}$. Moreover $|N(y_k) \cap N(w)| = n-y_k(n) =
n - (k+1)$. Therefore
\begin{eqnarray*} \eps_{y_k}(w) &=&
2(\ell(y_k) - |N(y_k)\cap N(w)|)\\
&=& k(k-1) + (n-k)(n-k-1) - 2(n-k-1)\\
&=& 2k^2 - 2k(n-1) + n^2-3n+2\\
&=& 2\left(k-\textstyle\frac{(n-1)}{2}\right)^2 +
\textstyle\frac{1}{2}(n^2-4n+3)\\
&=& 2\left(k-\textstyle\frac{(n-1)}{2}\right)^2 + \textstyle\frac{1}{2}(n-2)^2
-\textstyle\frac{1}{2}.
\end{eqnarray*}
If $n$ is odd, then this quantity is minimal when $k=\frac{n-1}{2}$.
Hence part (ii) holds. In this case, \begin{eqnarray*} e(w) &=&
\eps_{y_{(n-1)/2}}(w) = \textstyle\frac{1}{2}(n-2)^2
-\textstyle\frac{1}{2}\\ &=&
\lfloor\textstyle\frac{(n-2)^2}{2}\rfloor.\end{eqnarray*}
If $n$ is even, then $\eps_{y_k}$ is minimal when $k=\frac{n}{2}$ or
$k=\frac{n}{2}-1$. Hence we obtain part (iii). In either case, $e(w)
= \textstyle\frac{1}{2}(n-2)^2$. Combining the odd and even cases we
see that $e(w) = \lfloor\textstyle\frac{(n-2)^2}{2}\rfloor$. \end{proof}\\

\section{Zero Excess in Conjugacy Classes}\label{zero}

%\begin{prop}\label{}
%Let $X$ be a conjugacy class of $W$. Then there exists $w' \in X$
%such that $E(w^{\prime})=e(w^{\prime}) = 0$.
%\end{prop}

\textbf{Proof of Theorem 1.1} Suppose that $W$ is a Coxeter group,
$w \in \mathcal{W}$ and $X$ is the $W$-conjugacy class of $w$. Now
$w = r_1 \cdots r_{\ell}$ for certain $r_i \in R$ and some finite
$\ell = \ell(w)$. So $w \in \langle r_1, \dots, r_{\ell} \rangle$.
Thus it suffices to establish the theorem for $W$ of finite rank.
Accordingly we argue by induction on $|R|$. Suppose $K \subsetneqq
R$. If $X \cap W_K \ne \emptyset$, then by induction there exists
$w^{\prime} \in X \cap W_K$ with $e_{W_K}(w^{\prime}) = 0$, whence
$e(w^{\prime}) = 0$ and we are done. So we may suppose that $X
\cap W_K = \emptyset$ for all $K \subsetneqq R$. That is,
$X$ is a cuspidal class of $W$. \\

Choose $w \in X$. If $w = 1$ or $w$ is an involution, then $e(w) =
0$ by Lemma \ref{basics}(i). Thus we may suppose $w = xy$ where
$x$ and $y$ are involutions. By Theorem \ref{involutions} we may
conjugate $w$ so that $y \in W_J$ for some $J \subseteq R$, with
$y$ acting as $-1$ on $\Phi_J$. Thus $y \in Z(W_J)$. Now choose
$z$ to have minimal length in $\{g^{-1}xg \; | \; g \in W_J\}$. So
we have $z =
g^{-1}xg$ for some $g \in W_J$.\\

Suppose for a contradiction that there exists $r \in J$ with
$\ell(zr) < \ell(z)$. Since $z$ and $r$ are involutions,
$$\ell(rz) = \ell((rz)^{-1}) = \ell(zr) < \ell(z).$$ Applying
Proposition \ref{length} yields that either $rzr = z$ or
$\ell(rzr) = \ell(z) -2 < \ell(z)$. Since $r \in W_J$, the latter
possibility would contradict the minimal choice of $z$. Hence $rzr
= z$. So $r\cdot (z\cdot a_r) = z\cdot (r\cdot a_r) = -z\cdot
a_r$. It is well known that the only roots $\beta$ for which
$r\cdot \beta = -\beta$ are $\alpha_r$ and $-\alpha_r$. Thus
$z\cdot \alpha_r = \pm \alpha_r$. By assumption $\ell(zr) <
\ell(z)$ and therefore $z\cdot \alpha_r \in \Phi^-$ by Lemma
\ref{millie}(i), whence $z\cdot \alpha_r = -\alpha_r$. Combining
this with $y\cdot \alpha_r = -\alpha_r$ we then deduce that
$zy\cdot \alpha_r = \alpha_r$. Then Theorem \ref{tits} gives that
$zy$ is in a proper parabolic subgroup of $W$. Noting that $zy =
g^{-1}xgy = g^{-1}xyg = g^{-1}wg$, as $g \in W_J$, we infer that
$X$ is not a cuspidal class, a contradiction. We conclude
therefore that $\ell(zr) > \ell(z)$ for all $r \in J$.
Consequently $N(z) \cap \Phi^+_J = \emptyset$. Since $N(y) =
\Phi_J^+$ we deduce that $N(z) \cap N(y) = \emptyset$ and hence,
using Lemma \ref{lengthadd}, that $\ell(zy) = \ell(z) + \ell(y)$.
Setting $w_{\ast} = zy = g^{-1}wg$ we have $w_{\ast} \in X$ and
$e(w_{\ast}) = 0$, so completing the proof of Theorem \ref{exc=0}.
\qed

\newpage

%Sarah Hart \thanks{\noindent School of Economics, Mathematics and
      %Statistics,
%Birkbeck College, Malet Street, London, WC1E 7HX. \hspace*{4mm}
%s.hart@bbk.ac.uk }}

    \begin{center}{\Large\textbf{Corrigendum\\
    Involution Products in Coxeter Groups\\
    J. Group Theory 14 (2011), no.2, 251 - 259}}\\

      S.B. Hart and P.J. Rowley

\end{center}
\thispagestyle{empty}

In [1], Theorem 2.4 states a well-known result on Coxeter
groups which gives conditions under which the stabilizer of a
nonzero vector is a proper parabolic subgroup. However the
hypothesis of this result is incorrectly stated in our paper: it
holds for finite Coxeter groups but is not true in general for
infinite Coxeter groups. We are grateful to an anonymous referee of
a subsequent paper for pointing this out. As a consequence, the
proof of Theorem 1.1 in [1], which uses Theorem 2.4, is
incomplete. Here we complete the proof of Theorem 1.1
without recourse to Theorem 2.4.\\

Theorem 1.1 states that if $X$ is a strongly real conjugacy class of
a Coxeter group $W$ (not necessarily finite), then there exists
$w_{\ast} \in X$ such that $e(w_{\ast}) = 0$. That is to say, there
are involutions $\sigma$, $\tau$ of $W$ such that $w_{\ast} =
\sigma\tau$ and $\ell(w) = \ell(\sigma) + \ell(\tau)$. At the point
in the proof where Theorem 2.4 is used, we have established that
$zy$ is an element of $X$ where $z$ and $y$ are involutions with the
following properties. First, $y$ is the central involution of some
standard parabolic subgroup $W_J$ of $W$. The involution $z$ has the
property that $\ell(gzg^{-1}) \geq \ell(z)$ for all $g \in W_J$. It
follows that if $\ell(zr) < \ell(z)$ for
any $r \in J$, then $rzr = z$ and $z\cdot\alpha_r = -\alpha_r$.\\

Now let $K = \{r \in J: \ell(zr) < \ell(z)\}$.  From the above we
know that $z\cdot \alpha_r = -\alpha_r$ for all $r \in K$. If $K$ is
nonempty then, as $\Phi^+_K \subseteq N(z)$, $\Phi^+_K$ is finite.
Therefore $W_K$ has a unique longest element $w_K$, which is an
involution, and $N(w_K) = \Phi^+_K$. If $K = \emptyset$ we set $w_K
= 1$. In all cases, since $y$ is central in $W_J$ and $w_K \in W_J$,
we see that $w_Ky = yw_K$ is an involution. Moreover $zr = rz$ for
all $r \in K$, and thus $zw_K$ is also an involution. Let $\sigma =
zw_K$ and $\tau = w_Ky$. Then $\sigma\tau = zy \in X$. Moreover $z$
and $y$ both act as $-1$ on $\Phi^+_K$. Thus, by Lemma 2.2,
$$N(\sigma) = N(z) \setminus \left[-z\cdot N(w_K)\right] = N(z) \setminus N(w_K)$$
and $$N(\tau) = N(y) \setminus \left[ -y\cdot N(w_K)\right] = N(y) \setminus N(w_K)
= \Phi^+_J \setminus N(w_K).$$ Consider $r \in J$. If $r \in K$,
then $\alpha_r \in N(w_K)$ and so $\alpha_r \notin N(z) \setminus
N(w_K) = N(\sigma)$. On the other hand if $r \in J\setminus K$ then
by definition of $K$,  $\alpha_r \notin N(z)$ and hence $\alpha_r
\notin N(\sigma)$, which is after all a subset of $N(z)$. Hence for
all $r \in J$ we have $\alpha_r \notin N(\sigma)$. This implies that
$N(\sigma) \cap \Phi^+_J = \emptyset$, because every positive root
in $\Phi^+_J$ is a positive linear combination of some $\alpha_r, r
\in J$. But $N(\tau) \subseteq \Phi^+_J$ and therefore $N(\sigma)
\cap N(\tau) = \emptyset$. Hence, by Lemma 2.2, $\ell(\sigma\tau) =
\ell(\sigma) + \ell(\tau)$. Setting $w_{\ast} = \sigma\tau$ we have
$w^{\ast} \in X$ and $e(w_{\ast}) = 0$, so completing the proof of
Theorem 1.1.\hfill $\square$

\section*{References}
[1] S. B. Hart and P. J. Rowley {\em Involution Products in Coxeter Groups},
 J. Group Theory 14 (2011), no.2, 251 - 259.

\end{document}